\documentclass[a4paper]{amsart}
\usepackage{amsmath,amsthm,amssymb,hyperref,mathrsfs}
\usepackage{aliascnt}
\usepackage{lmodern}
\usepackage[T1]{fontenc}

\usepackage[textsize=footnotesize]{todonotes}

\usepackage{xcolor}
\definecolor{dblue}{rgb}{0,0,0.70}
\hypersetup{
	unicode=true,
	colorlinks=true,
	citecolor=dblue,
	linkcolor=dblue,
	anchorcolor=dblue
}

\makeatletter
\expandafter\g@addto@macro\csname th@plain\endcsname{%
		\thm@notefont{\bfseries}
	}%
\expandafter\g@addto@macro\csname th@remark\endcsname{%
		\thm@headfont{\bfseries}
	}%
\makeatother

\newtheorem{theorem}{Theorem}[section]	
\newtheorem*{theorem*}{Theorem}

\newaliascnt{lemma}{theorem}
\newtheorem{lemma}[lemma]{Lemma}

\aliascntresetthe{lemma}
\newtheorem*{lemma*}{Lemma}

\newaliascnt{proposition}{theorem}
\newtheorem{proposition}[proposition]{Proposition}
\aliascntresetthe{proposition}

\newaliascnt{corollary}{theorem}

\aliascntresetthe{corollary}

\theoremstyle{remark}

\newaliascnt{remark}{theorem}
\newtheorem{remark}[remark]{Remark}
\aliascntresetthe{remark}
\newaliascnt{question}{theorem}

\aliascntresetthe{question}

\newtheorem*{question*}{Question}

\newaliascnt{definition}{theorem}
\newtheorem{definition}[definition]{Definition}
\aliascntresetthe{definition}

\newaliascnt{example}{theorem}

\aliascntresetthe{example}

\usepackage{enumitem}
\setlist[itemize]{leftmargin=6.5mm}

\renewcommand{\restriction}{\mathbin\upharpoonright}

\newcommand{\axiomft}[1]{\mathsf{#1}}

\newcommand{\CH}{\axiomft{CH}}

\newcommand{\PFA}{\axiomft{PFA}}

\newcommand{\MA}{\axiomft{MA}}
\newcommand{\MM}{\axiomft{MM}}

\newcommand{\TCG}{\axiomft{TCG}}

\DeclareMathOperator{\cf}{cf}
\DeclareMathOperator{\dom}{dom}

\newcommand{\forces}{\mathrel{\Vdash}}

\newcommand\PP{\mathbb{P}}
\newcommand{\power}{\mathcal{P}}

\newcommand{\BB}{\mathbb B}

\newcommand{\cB}{\mathcal B}

\newcommand{\cM}{\mathcal M}
\newcommand{\cN}{\mathcal N}
\newcommand{\cC}{\mathcal C}

\newcommand{\bd}{2-\mathrm{bd}}
\newcommand{\pto}{\to^{\text{poly}}}

\newcommand{\tup}[1]{\langle#1\rangle}

\newcommand{\normal}{\normalfont\normalsize}

\newcommand{\arxiv}[1]{\small\href{https://arxiv.org/abs/#1}{arXiv:#1}\normal}
\newcommand{\doi}[1]{\small\href{https://doi.org/#1}{DOI}\normal}

\title[MEHIFOX Notes]{Methods in Higher Forcing Axioms}

\begin{document}
\begin{abstract}
Methods of Higher Forcing Axioms was a small workshop in Norwich, taking place between 10--12 of September, 2019. The goal was to encourage future collaborations, and create more focused threads of research on the topic of higher forcing axioms. This is an improved version of the notes taken during the meeting by Asaf Karagila.
\end{abstract}
\maketitle
\section{Introduction}
David Asper\'o and Asaf Karagila hosted a small workshop in Norwich between 10--12 of September, 2019. Here the term ``workshop'' was not intended as it is usually perceived (i.e., a small conference with laser-focused talks) but rather a work-shop. The goal was to help and seed future research on forcing axioms, as well as to try and organise a rudimentary literature reading list on the various topics in higher forcing axioms.

The structure of each day was very loose: one person presented a problem, and a discussed ensued. This (scientific) report is based on the notes taken during the meeting by Asaf Karagila, and cleaned up later as a joint effort of the participants.

We (the organisers) feel that this very informal approach to the workshop is a very fruitful approach when trying to coalesce ideas from different parts of a research subject. We hope that this will encourage others to organise similar meetings to workshop similar topics in the future.

\subsection*{The Participants}
The following people attended, listed in alphabetic order.
\begin{itemize}
\item Prof.\ Uri \textbf{Abraham}, Ben-Gurion University of the Negev.
\item Dr.\ David \textbf{Asper\'o}, University of East Anglia.
\item Prof.\ Mirna \textbf{D\v{z}amonja}, University of East Anglia.
\item Dr.\ Asaf \textbf{Karagila}, University of East Anglia.
\item Dr.\ Miguel \'Angel \textbf{Mota}, Instituto Tecnol\'ogico Aut\'onomo de M\'exico.
\item Prof.\ Tadatoshi \textbf{Miyamoto}, Nanzan University.
\item Dr.\ Dorottya \textbf{Sziraki}, R\'enyi Alfr\'ed Matematikai Kutat\'oint\'ezet.
\item Prof.\ Boban \textbf{Veli\v{c}kovi\'c}, IMJ-PRG, Universit\'e de Paris.
\item Prof.\ Teruyuki \textbf{Yorioka}, Shizuoka University.
\end{itemize}

\subsection*{Acknowledgements}
This meeting would not have taken place without the generous support of the London Mathematical Society (Research Workshop Grant award no.~WS-1819-06) and the School of Mathematics at the University of East Anglia. We also want to thank Chris Lambie-Hanson and Tanmay Inamdar for their kind permission to share some of their work here.
\newpage

\section{Problem 1: David Asper\'o}
If we have $2$-types of models side-condition, as in the Neeman idea, then we have models which are countable and of size $\aleph_1$, or transitive in the proof of $\PFA$. 

So we can have a situation $M_0\in N\in M_1$, where $M_0\notin M_1$. But we do require a finite chain of models such that each model is an element of the next, and the chain is closed under intersections. Equivalently, we can assign each model in the chain an ordinal, $\delta_i=\sup(Q_i\cap\omega_2)$. Then the sequence of $\delta_i$ is increasing, and if $i<j$, $Q_i\notin Q_j$, then there is some $k$, $i<k<j$ such that $Q_k\in Q_j$ and for some $\ell<i$, $Q_\ell=Q_j\cap Q_k$.

The idea is that if $Q_i\notin Q_j$, then $Q_j$ is a small model, and there is a large model between them, and the intersection of $Q_j$ and the large model is below $Q_i$ in the chain. It might not be an element of $Q_i$, in which case there is another large model, etc.

There is a problem with ``the obvious generalisation to models of 3-types'', as Boban presents below. If we have 3-types of models and we require the chains to be closed under arbitrary intersections---as in the case of 2-types---then the natural proof of cardinal preservation breaks down. In fact, the resulting forcing in this case will not be strongly proper for at least one of the relevant classes of models.

The above configuration makes crucial use of the side conditions being closed under arbitrary intersections. Hence, the following possibility does not seem to be ruled out: consider the forcing $\PP$ of finite chains of models of 3-types satisfying the following \textit{weak closure under intersections}: for $N\in M$ in the chain and $|M|<|N|$, $N\cap M$ is in the chain. Is $\PP$ strongly proper for the relevant class of models with this property?

\begin{remark}[Added Nov.~2019] In the meantime David found a configuration which strongly suggests that, unfortunately, this forcing $\PP$ is not strongly proper either (for some of the relevant models, at least). It seems that forcing with 3-types of models, as opposed to 2-types, is bound to collapse cardinals (at least when considering the ``natural approach'').

We use $M$ for small models (say countable models), $N$ for middle-sized models (say of size $\aleph_1$) and $P$ for big models (for example of size $\aleph_2$). We say that a model $Q_1$ is \textit{above} a model $Q_0$ (or $Q_0$ is \textit{below} $Q_1$) if $\sup(Q_0\cap\omega_3)<\sup(Q_1\cap\omega_3)$. We assume that all models are closed under sequences of length less than their size.

Suppose we are proving properness for a model $M$ in some condition $s$. In $M$ we are reflecting $s$, which involves adding a model $P\in M$. Above $M$ there are models $M_1$ and $M_0$ in $s$, $M_1$ above $M_0$, both containing $P$, which means we will need to add $M_1\cap P$ and $M_0\cap P$ to the final chain. Although $M_1$ is above $M_0$, it turns out that $M_0$ is not a member of $M_1$, and in fact there is some $N\in  M_1\cap s$ above $M_0$ such that $M_1\cap N$ is below $M_0$. Suppose also that $M_0$ does not belong to $N$ and in fact there is some $P_0\in N\cap s$ above $M_0$ such that $N\cap P_0$ is below $M\cap P$ (note that $M\cap P$ is below $M_0\cap P$). Also, let us assume that $M$ is not in $M_1\cap N$ as there is some $P_1\in M_1\cap N\cap s$ such that $M_1\cap N\cap P_1$ is below $N\cap P_0$. 

Now, it seems that $M_0\cap P$ is not be a member of $M_1\cap P$, but then we need some middle-sized or big model model $Q$ in our chain such that $Q\in M_1\cap P$ and $M_1\cap P\cap Q$ is equal to or below $M_0\cap P$. But, given the above configuration---which can be certainly cooked up--- it is not clear where such a model could come from. 
\end{remark}
\newpage 
\subsection{Example of failure for 3-types of models (presented by Boban)}
Suppose that things worked fine with $3$-types models, i.e.\ there is closure under intersection and the forcing is strongly proper. We have models of size $\aleph_0,\aleph_1$, and $\aleph_2$.

If $G$ is a generic filter, then, let $P_i$ for $i<\omega_3$ be the sequence of models of size $\aleph_2$. Let $N$ be a model of size $\aleph_1$ above the first $P_\xi$, for $\xi<\omega_1$, and let $M$ be a countable model above $N$ such that $N\in M$. We let $\delta$ be large enough, and $P_\delta\notin M$.

Compare now $N\cap P_\delta$ and $M\cap P_\delta$. But now we have that $M\cap P_\delta\in N\cap P_\delta$. However, $\sup(N\cap\omega)<\sup(M\cap\omega_2)$ while $\sup(M\cap P_\delta\cap\omega_2)<\sup(N\cap P_\delta\cap\omega_2)$.

But since $\omega_2\subseteq P_\delta$, as $P_\delta$ has size $\aleph_2$, the intersection $M\cap P_\delta\cap\omega_2=M\cap P_\delta$, and same for $N$. And this is impossible.

\section{Problem 2: Tadatoshi Miyamoto}
\subsection{Motivation: Asper\'o--Mota type forcing axioms}
These are forcing axioms that deal with $\aleph_1$ dense subsets for a forcing $\PP$, which typically is of size $\aleph_1$, with finite conditions. And we can iterate using side conditions which are countable elementary submodels of $(H(\aleph_2),\in,\dots)$ up to $2^{\aleph_0}=\aleph_2$. 

We get some sort of weak negation of club guessing, depending on the variables. The forcings in this class include c.c.c.\ and other things.

Now we want to have a ``higher'' analogue. So we want to meet $\aleph_2$ dense open sets. Typically, now, we want to work with forcings of size $\aleph_2$, whose conditions are countable. The models now are $N\prec(H(\aleph_3,\in,\dots)$, such that $|N|=\aleph_1$, and $\omega_1\subseteq N$, and $N$ is $\sigma$-closed. This will provide us with $\lnot\TCG(S^2_1)$, where $\TCG$ is ``tail club guessing''.

If our class of forcing contains $\omega_1$-closed and $\omega_2$-c.c.\, then there is no forcing axiom for this class. But Shelah provided a subclass of these which does admit a forcing axiom.

\begin{definition}[$\TCG(S^2_1)$]
There is a sequence $\tup{\eta_\delta\mid\delta\in S^2_1}$ such that $\eta_\delta$ is a closed copy of $\omega_1$, cofinal in $\delta$ such that for every club $C\subseteq\omega_2$, there is $\delta\in S^2_1$ and $\delta'<\delta$ such that $\eta_\delta\setminus\delta'\subseteq C$.
\end{definition}
\begin{theorem}
$\diamondsuit(S^2_1)\implies\TCG(S^2_1)$.
\end{theorem}
\begin{theorem}
$\PP$ preserves $\TCG(S^2_1)$ if one of the two conditions hold:
\begin{enumerate}
\item $\PP$ preserves $\omega_1$ and has $\aleph_2$-c.c., or
\item $\PP$ is $\aleph_2$-closed.
\end{enumerate}
\end{theorem}
\subsection{Questions}\hfill

\textbf{Question 1:} Suppose that $2^{2^{\aleph_0}}=\aleph_2$. Does that imply $\TCG(S^2_1)$?

\textbf{Question 2:} We already know that a forcing axiom for the single poset $\PP_{\text{ff}}^c$ (see definition below), then $\CH+2^{\omega_1}=\omega_3$, but $\TCG(S^2_1)$ fails. Can we have a forcing axiom for Shelah's class of forcings adjoined by $\PP_{\text{ff}}^c$?
\subsection{\texorpdfstring{$\PP^c_{\text{ff}}$}{PPff}}
We want to add an $\omega_1$-club to $S^2_1$. A condition $p\in\PP^c_{\text{ff}}$ is a function $p\colon S^2_1\to\omega_2$ satisfying the following properties:
\begin{enumerate}
\item $\dom p$ and it is a countable subset of $S^2_1$.
\item For every $\delta\in\dom p$, $\delta\leq p(\delta)$.
\item For every $\delta',\delta\in\dom p$ then if $\delta'<\delta$, we have $p(\delta')<\delta$.
\end{enumerate}
We say that $q\leq p$ if $p\subseteq q$. The generic club, $E$, is $\bigcup\{\dom p\mid p\in G\}$, where $G$ is a generic filter.

This forcing is strongly $\sigma$-closed, i.e.\ it admits greatest lower bounds for descending sequences. Moreover, this forcing preserves $\omega_2$ and therefore all cardinals.

\textbf{Point 1:} If $p\forces\dot E\cap\delta$ is cofinal in $\delta\in S^2_1$, then $\delta\in\dom p$ and $p\forces\delta\in\dot E$.

\textbf{Point 2:} $\forces$ If $\delta\in S^2_1$ and $\dot E\cap\delta$ is cofinal in $\delta$, then we have to avoid an unbounded subset of every ground model club in $\delta$.
\subsection{Forcing axioms for meeting \texorpdfstring{$\aleph_2$}{aleph 2} dense sets}
We want to have a forcing axiom allowing us to meet $\aleph_2$ dense open sets for the class including $\PP^c_{\text{ff}}$, the Baumgartner forcing, standard forcing for adding morasses, etc.

Such a forcing axiom implies the failure of $\TCG(S^2_1)$, and therefore the failure of $V=L$. Nevertheless, the inclusion of the morass forcing in the class means that we get an $(\omega_1,1)$-morass. Therefore $\square_{\omega_1}$ holds, and $\CH$ holds, there are Kurepa trees and $\omega_2$-Suslin trees.

\begin{center}
\begin{tabular}{|l|l|l|l|}
\hline
Property & $\PP_{\text{ff}}^c$ & $\BB$ & $\PP^0_{\text{morass}}$ \\\hline
strongly $\sigma$-closed & Yes & Yes & Yes\\\hline
well-met & Yes & Yes & No\\\hline
$\aleph_2$-c.c. & No & Yes & Yes\\\hline
\end{tabular}
\end{center}
\textbf{Question 3:} Assuming the forcing axiom above, we can prove that Shelah's forcing which is $\sigma$-lattice\footnote{Every countable set of pairwise compatible conditions have a greatest lower bound.} and $\aleph_2$-stationary Knaster is in fact equivalent to Baumgartner's forcing. Is it consistent that this is not the case with some weak forcing axiom?

\begin{definition}
We say that a strongly $\sigma$-closed forcing $\PP$ is \textit{Norwich} if for every $p$ we can assign a structure $A^p=\tup{U^p,<,F^p,f^p,R^p}$ such that:
\begin{enumerate}
\item $U^p$ is a countable, non-empty, subset of $\omega_2$,
\item $F^p$ is a countable subset of $\mathcal P(U^p)$,
\item $f^p$ is a partial function from $U^p$ to itself, such that $i<j$ and $i,j\in\dom f^p$, then $f^p(i)<j$,
\item $R^p\subseteq U^p\times U^p$.
\end{enumerate}
Such that $p\mapsto A^p$ is:
\begin{enumerate}
\item injective,
\item monotonous: if $q\leq p$, then $A^p$ is a pointwise subset of $A^q$, i.e.\ $U^p\subseteq U^q$, etc.,
\item continuous: i.e.\ if $p_n\to p$, then $U^p=\bigcup_{n<\omega}U^{p_n}$, etc.,
\item amalgamation property: if $A^p\sim A^q$, then $p\mathrel{\|}q$.
\item pre-generic: if $\Delta\in[\omega_2]^{\leq\omega}$ such that $\sup U^p<\min\Delta$, then there is $q\leq p$ such that $\Delta\subseteq\dom f^q$.
\end{enumerate} 
\end{definition}
We know that $\BB$ is Norwich with an extra pre-generic condition, which itself is Shelah. Can we separate those notions? Can we combine a forcing axiom for both Norwich and Shelah forcings, or does this lead to a contradiction?

\textbf{Question 4:} \textit{(Added Nov.~2019)} Is it possible to use two-types of models (countable, $\aleph_1$ with a reasonable extra condition) with finite symmetric systems \`a~la Asper\'o--Mota for a larger continuum, say, $\aleph_3$, with the forcing axiom for the single partial order $\PP^c_{\text{ff}}$?
\section{Problem 1.5: David Asper\'o}
The following is due to Chris Lambie-Hanson. Let $\kappa$ be a regular uncountable cardinal, say that a sequence $\tup{f_\alpha\mid\alpha<\kappa}$ is a \textit{coherent $\kappa$-scale} if 
\begin{enumerate}
\item $f_\alpha\colon\omega\to\omega$,
\item $f_\alpha<^* f_\beta$ for all $\alpha<\beta<\kappa$, and
\item there is a sequence $\tup{C_\alpha\mid\alpha\in S^\kappa_{\geq\omega_1}}$ such that $C_\alpha\subseteq\alpha$ is a club,
and there is $\tup{k_\alpha\in\omega\mid\alpha\in S^\kappa_{\geq\omega_1}}$ such that for all $\delta\in S^\kappa_{\geq\omega_1}$ and for all $\beta\in C_\delta$, if $n>k_\delta$, then $f_\beta(n)<f_\delta(n)$.
\end{enumerate}
\begin{theorem}[Lambie-Hanson]
It is consistent (modulo large cardinals, which are probably unnecessary) to have a coherent $\omega_2$-scale. However, it is provable that there is no coherent $\omega_4$-scale.
\end{theorem}
\textbf{Question:} Is this consistent to have a coherent $\omega_3$-scale?

This would be a natural question to solve with a $3$-types side condition approach.

One related result is due to Shelah,\footnote{\href{https://doi.org/10.1002/malq.200910010}{Shelah, On long increasing chains modulo flat ideals. (Math.\ Logic.\ Quart.\ 56(4), 397--399.)}} below is a related result due to Tanmay Inamdar.
\begin{theorem}[Inamdar]
There is no sequence $\tup{X_\alpha\mid\alpha<\omega_4}$ of sets in $[\omega_3]^{\omega_3}$ which is $\subseteq^*$ (modulo countable) such that: $\alpha<\beta$ implies $X_\alpha\setminus X_\beta$ is countable, and $X_\beta\setminus X_\alpha\in[\omega_3]^{\omega_3}$.
\end{theorem}

Boban proposes the following observation.
\begin{proposition}
Assume Chang Conjecture holds. Then there is no sequence $\tup{f_\alpha\mid\alpha<\omega_2}$ of functions $f_\alpha\colon\omega_1\to\omega_1$ which are increasing modulo finite.
\end{proposition}
\begin{proof}
Let $I$ be some set, $g\colon I\to\omega_1$ and $\eta<\omega_1$, $\min(g,\eta)$ is the truncation of $g$ at $\eta$. For every $\alpha<\omega_2$, and $\xi,\eta<\omega_1$, $f_\alpha^{\xi,\eta}$ is $\min(f_\alpha\restriction[\xi,\xi+\omega),\eta)$.

Fixing $\xi$ and $\eta$ for a moment, if $\alpha<\beta$, $f_\alpha^{\xi,\eta}\leq^*f_\beta^{\xi,\eta}$. We define clubs, $C^{\xi,\eta}$: If $f_\alpha^{\xi,\eta}$ for $\alpha\to\omega_2$ stabilises at some point, we let $C^{\xi,\eta}$ be the tail on which it is stable of the form $\omega_2\setminus\mu$. If the sequence does not stabilise, then we choose $C^{\xi,\eta}$ such that $\alpha<\beta\in C^{\xi,\eta}$ implies $f_\alpha^{\xi,\eta}\neq_{\text{fin}} f_\beta^{\xi,\beta}$, i.e.\ $\{n<\omega\mid f_\alpha(\xi+n)<f_\beta(\xi+n)\leq\eta\}$ is infinite.

Let $C=\bigcap_{\xi,\eta<\omega_1}C^{\xi,\eta}$ be a club in $\omega_2$. And let $c\colon[C]^2\to\omega_1$ be defined by $c(\alpha<\beta)=\max\{\xi\mid f_\alpha(\xi)\geq f_\beta(\xi)\})$. By CC, there is $S\subseteq C$, $S=\{\alpha_\rho\mid\rho<\omega_1\}$, and $c``[S]^2\subseteq\xi<\omega_1$. If $\rho<\tau<\omega_1$, then $f_{\alpha_\rho}<f_{\alpha_\tau}$ in the interval $[\xi,\xi+\omega)$. Let $\eta$ be large enough such that the truncation by $\eta$ on $[\xi,\xi+\omega)$ is irrelevant for $f_{\alpha_0}$ and $f_{\alpha_1}$.

This witnesses that $C^{\xi,\eta}$ was defined by picking witnesses for non-stabilisation. This means now that for every $\alpha<\beta\in C^{\xi,\eta}$, the set $\{n\mid f_\alpha^{\xi,\eta}(\xi+n)<f_\beta^{\xi,\eta}\leq\eta\}$ is infinite. This means for some $n<\omega$, $f_{\alpha_i}(\xi+n)<f_{\alpha_\beta}(\xi+n)<\eta$. But this is impossible.
\end{proof}

Boban presents a proof of the claim by Chris Lambie-Hanson that there is no coherent $\omega_4$-sequence.
\begin{theorem}
There is no coherent $\omega_4$-scale.
\end{theorem}
\begin{proof}
We adapt the proof of the following PCF theorem lemma.
\begin{lemma}
Suppose that there is a function $F\colon\power(\theta)\to\theta+1$ such that:
\begin{enumerate}
\item $X\subseteq Y$ implies $F(X)\leq F(Y)$.
\item $F(X)\geq\sup X$.
\item For every $\delta\geq\cf(\delta)\geq\omega_1$, there is a club $C_\delta\subseteq\delta$ such that $F(C)=\delta$.
\item For all $X$, there is a countable $X_0\subseteq X$ with $F(X)=F(X_0)$.
\end{enumerate}
Then $\theta\neq\omega_4$.
\end{lemma}
In the PCF context, $F$ is really a closure operator of the PCF, and $F$ as here is the maximum. In the case of a coherent $\omega_4$-scale, we can do the following.

Suppose that $\tup{f_\alpha\mid\alpha<\theta}$ is a coherent $\theta$-scale. We define the following $F\colon\power(\theta)\to\theta+1$. If $X\subseteq\theta$, and there are infinitely many $n<\omega$ such that $\sup\{f_\alpha(n)\mid\alpha\in X\}=\omega$, define $F(X)=\theta$. Otherwise, there is some function $f\colon\omega\to\omega$ which is the pointwise maximum on a cofinite subset of $\omega$. If there is some $\delta$ such that $f\leq^*f_\delta$, then $F(X)=\delta$. And if there is no such $\delta$, then $F(X)=\theta$ again.

Now, due to the definition of a coherent $\theta$-scale, the three requirements of the lemma hold quite easily. Therefore $\theta<\omega_4$.
\end{proof}
\begin{proof}[Proof of PCF-style lemma]
Suppose otherwise, so $\theta=\omega_4$. We have a club guessing sequence on $\omega_3$, $C_\xi$ for $\xi\in S^3_1$. We will define $C$, a closed subset of $\omega_4$ using the function $F$ and the club guessing sequence.

Let $\pi$ be our partial isomorphism from an initial segment of $\omega_3$ into $C$. Suppose we defined $\pi$ and $C$, up to $\xi$. Then $\alpha_{\xi+1}=\sup(\{F(\pi``(C_\eta\cap\xi))\mid\eta\geq\xi\}\cap\omega_4)+1$.

Namely, we look at the remaining $S_\eta\cap\xi$, above $\xi$, copy them to $\omega_4$ using $\pi$, apply $F$ to each one, and consider all those which do not give $\omega_4$ under $F$. This is well-defined, since $\pi$ was defined up to $\xi$, so $C_\eta\cap\xi$ is defined. Now there is a club $C\subseteq\omega_3$ such that $F(\pi``C)=\delta=\sup\{\alpha_\xi\mid\xi<\omega_3\}$.

Now we have a contradiction, since there is a countable $C_0\subseteq C$ such that $F(C_0)=\delta$. That means there is $\xi\in S^3_1$ such that $C_\xi\subseteq C$. Now, $F(\pi``C_\xi)\leq\delta$, and for every $\rho\in C_\xi$, $F(\pi``(C_\xi\cap\rho+1))<\min\pi``(C_\xi\setminus\rho+1)$, this is because of the way we defined $\pi$. But this is impossible.
\end{proof}
David notes that for a coherent $\omega_3$-scale to exist, it must be the case that club guessing on $S^2_1$ fails. This is possible, of course, but puts a limitation.

David proposes the following natural question, which Boban notes of having considered in 1998.

\textbf{Question:} Suppose that $\aleph_\omega$ is a strong limit cardinal. Is there a forcing $\PP$ such that $|\PP|<\aleph_\omega$, $\PP$ preserves $\aleph_1,\aleph_2$ and $\aleph_3$, and $\PP$ forces club guessing on $S^2_1$?

The existence of such $\PP$ implies $2^{\aleph_\omega}<\aleph_{\omega_3}$ in that model. So a positive answer would improve the bounds on the PCF Conjecture.

David continues. Suppose that there is a sequence $\{C_\xi\mid \xi<\omega_2\}$ of clubs of $\omega_1$ such that every club $C\subseteq\omega_1$ contains one of them, i.e.\ the density of the club filter is $\aleph_2$, then you have a form of weak club guessing given below.

\begin{definition}
$\operatorname{CG}(S^2_1)^-$ means that there is a sequence $(\cC_\alpha\mid\alpha\in S^2_1)$ such that $|\cC_\alpha|\leq\aleph_2$, $\cC_\alpha$ is a set of clubs of $\alpha$, and for every club $C\subseteq\omega_2$, there is some $\alpha$ and $D\in\cC_\alpha$ such that $D\subseteq C$.
\end{definition}
In Abraham--Shelah's model\footnote{\href{https://doi.org/10.2307/2273954}{Abraham and Shelah, On the intersection of closed unbounded sets. (J.~Sym.~Log.\ 51(1) 180--189.)}} where there is a family $\{C_\alpha\mid\alpha<\omega_2\}$ of clubs of $\omega_1$, such that every outer model with the same $\omega_1$ every uncountable subfamily has a finite intersection. So in their model there is no ``nice'' forcing which forces that the density of the club filter on $\omega_1$ is $\aleph_2$. This gives a negative answer to a strong version of the question, but not the question itself.
\section{Problem 3: Uri Abraham}
\subsection{Motivation: Basic results in polychromatic Ramsey theory}
In usual Ramsey theory we ask for homogeneous sets. In polychromatic Ramsey theory we want the opposite.
\begin{definition}
If $\kappa$ is a cardinal and $f\colon[\kappa]^2\to\kappa$ is some function, we say that $X\subseteq\kappa$ is \textit{polychromatic} if $f\restriction[X]^2$ is injective.
\end{definition}
We will generally require some properties to hold for $f$, else we might not have large polychromatic sets. We take the convention that in a pair $\{\alpha,\beta\}$, it is always the case that $\alpha<\beta$.
\begin{enumerate}
\item $f$ is normal if when $\{\alpha_1,\beta_1\},\{\alpha_2,\beta_2\}\in[\kappa]^2$, if $\beta_1\neq\beta_2$, then $f(\alpha_1,\beta_1)\neq f(\alpha_2,\beta_2)$.
\item $f$ is $<\omega$-bounded if any $\nu<\kappa$, $f^{-1}(\nu)$ is a finite set.
\end{enumerate}
We will write $\kappa\pto(\lambda)^2_{\bd}$ to mean that if $f\colon[\kappa]^2\to\kappa$ is normal\footnote{Normality is not really needed, but it simplifies the presentation.} and $2$-bounded, then there is a polychromatic set of order type $\lambda$.

\begin{theorem}
$\omega_1\pto(\alpha)^2_{\bd}$ for every $\alpha<\omega_1$.
\end{theorem}
\begin{proof}
Given $d\colon[\omega_1]^2\to\omega_1$ which is normal and $2$-bounded, take $M_i\prec H(\omega_2)$ countable and increasing such that $d\in M_0$, where $i\leq \alpha$. Enumerate $\alpha+1$ as $\{\alpha_n\mid n<\omega\}$, with $\alpha_0=\alpha$ as the top point. We now construct our polychromatic set by recursion on $\omega$. First we pick $x_0\in \omega_1\setminus M_\alpha$.

By the $n$th step for $n\geq 1$, suppose we took care of $n$ points, $F=\{x_i \mid i<n\}$, such that $x_i\in M_{\alpha_i+1}\cap \omega_1\setminus M_{\alpha_i}$. We want to add $x_n$ to our set, and we want it to lie in $M_{\alpha_n+1}\setminus M_{\alpha_n}$. $x_0$ is an evidence to show that in $M_{\alpha_n+1}$ there is an unbounded set of points, $\beta$, which can (just like $\alpha_0$) be added to $\{x_i\mid\alpha_i<\alpha_n\}$ and form a polychromatic set.

Which $\beta$ will we choose in this set (and inside $M_{\alpha_n+1}\setminus M_{\alpha_n}$)? There are finitely many points which lie in $F$ and above $M_{\alpha_n}$, and they obtain finitely many colors with the points of $F$ below $M_{\alpha_n}\cap \omega_1$. So we can find a large enough $\beta'$ to work as our $x_n$ and to ensure that $F\cup\{x_n\}$ is polychromatic.

Since we built $X$ from intervals between the models, we get it to have an order type of $\alpha+1$.
\end{proof}
Some remarks on the proof. It can be easily generalised assuming $\CH$, or Chang Conjecture on some uncountable cardinals. 

\begin{theorem}
$\PFA\implies\omega_1\pto(\omega_1)^2_{\bd}$.
\end{theorem}
The proof below is due to Uri Abraham and James Cummings, the result is due to Stevo Todorcevic.
\begin{proof}
Let $d\colon[\omega_1]^2\to\omega_1$ be a normal and $2$-bounded function. We define a partial order $\PP$, whose conditions are $t=(p,\cM)$ such that:
\begin{enumerate}
\item $p$ is a finite polychromatic set.
\item $\cM=(M_0\in M_1\in\dots\in M_n)$, and $M_i\prec H(\omega_2)$ is a countable elementary submodel.
\item We denote by $\delta_i^t=M_i\cap\omega_1$, and if $t$ is clear from context, we omit it from the superscript. Now we require that $p\cap(\delta_i,\delta_{i+1})$ is at most a singleton for all $i<|p|$.
\footnote{We may assume, by density arguments, that $|p|=n$, but it is kind of irrelevant.}
\end{enumerate}
The order on $\PP$ is given by pointwise inclusion, i.e.\ $(p,\cM)\leq (p',\cM')$ if $p'\subseteq p$ and $\cM'\subseteq\cM$.

We want to prove now that $\PP$ is a proper partial order. Suppose that $s=(p_s,\cM_s)$ is a condition such that for some $K\in\cM_s$, $K=N\cap H(\omega_2)$ where $N\prec H(\theta)$ of a large enough $\theta$ such that $\PP\in N$. We want to prove that $s$ is an $N$-generic condition.

Let $D\in N$ a dense subset of $\PP$. By extending $s$ if necessary, we may assume that it is already in $D$. In $N$ we want to extend $s\restriction N=(p_s\cap N,\cM_s\cap N)$ into a condition of $D\cap N$ which is compatible with $s$. 

In order to express the idea of the proof without too many indexes, we take a simple but illustrative case and assume that $\cM=(M_0,\dots, M_n, M_{n+1},M_{n+2})$ where $M_n=K=N\cap H(\omega_2)$, and $p_s= p_s\cap K\cup \{x_{n+1},x_{n+2}\}$ where $x_n\in \delta_{n+1}\setminus \delta_n$ and $x_{n+1}\in \delta_{n+2}\setminus \delta_{n+1}$. $p_s=p_s\cap K\cup \{x_n,x_{n+1}\}$.

We want to extend $s\restriction K$ to some condition in $D\cap K$ that remains compatible with the part of $s$ beyond $K$. The problem is ensuring that the working part of the extension remains polychromatic even with the addition of the two external ordinals $x_n$ and $x_{n+1}$.

Work now in $N$, this is a structure which is aware of $\PP$ and the forcing relation. Define $B$ to be the set of pairs $\tup{(\mu_1,y_1),(\mu_2,y_2)}$ for which there is a condition $r\in D$ with $r\restriction K = s\restriction K$, and $r$ contains two additional models and two additional ordinals $y_1$ and $y_2$ where $(\mu_1,y_1)$ and $(\mu_2,y_2)$ play in $r$ the same role that $(\delta_n,x_n)$ and $(\delta_{n+1},x_{n+1})$  play in $s$. 

As $B$ is definable, we have $B\in N$, and since $B\in H(\omega_2)$, $B\in M_n$. $B$ is uncountable since $\tup{(\delta_1,x_1),(\delta_2,x_2)}$ is an evidence that the set is not bounded. In fact, $B$ is `richer': there is an uncountable set of pairs $(\mu_1,y_1)$ for each one of which there is an uncountable set of pairs $(\mu_2,y_2)$ such that $\tup{(\mu_1,y_1),(\mu_2,y_2)}\in B$. (First argue that there are uncountably many $(\mu_2,y_2)$ such that  $\tup{ (\delta_1,x_1) ,(\mu_2,y_2)}\in B$.)

Then we can again ensure that there is some $\tup{(\mu_1,y_1),(\mu_2,y_2)}\in B$, by the fact that $B$ is rich, such that the ordinals $\{y_1,y_2\}$ when added to $p_s$ is polychromatic. Now using $\PFA$ we get a polychromatic set of type $\omega_1$ by picking a generic meeting all conditions with models of at least height $\alpha$ as the $\alpha$th dense open set.
\end{proof}

\begin{theorem}
$\MM$ implies that if $f\colon[\omega_2]^2\to\omega_2$ is normal is $<\omega$-bounded, then there is a \textit{closed} polychromatic set of order type $\omega_1$.
\end{theorem}
\begin{proof}[Sketch of Proof]
The idea is similar to the above. Here the conditions are $(\cM,\cB)$ where $\cM$ is a finite chain of countable elementary submodels of $H(\omega_3)$ such that the set $\{\sup(M_i\cap\omega_1)\mid i<|\cM|\}$ is polychromatic, and $\cB$ is a set of disjoint intervals in $\omega_2$ which are also disjoint of the polychromatic set.
\end{proof}
\subsection{Question and ideas}
Stevo Todorcevic proved the consistency of this polychromatic principle without using large cardinals. But we can also show that Martin's Axiom alone is not enough to prove the polychromatic result.

\vspace{2em}
\textbf{Question 1:} Is there a forcing axiom strengthening Martin's Axiom, which can prove the theorem, which does not require large cardinals?

\textbf{Ideas:}
Asper\'o--Mota style forcing axioms might work. It seems that perhaps $\MA^{1.5}_{\aleph_2}(\text{layered})$ might be a suitable candidate for this. Some relevant ideas. Let $\cN$ be a finite set of countable substructures, we say that $\cN$ is \textit{layered} if for all $N_0,N_1\in\cN$, if $\delta_{N_0}<\delta_{N_1}$, then there is $N\in\cN$ such that $\delta_N=\delta_{N_1}$ and $N_0\in N$.

Say now that a forcing $\PP$ has $\aleph_{1.5}$-c.c.\ with respect to layered families if for every large enough $\theta$, there is a club of $D\subseteq [H(\theta)]^{\omega}$ such that for all $\cN\subseteq D$ which is finite and layered, for every $p\in N\in\cN$ with $\delta_N$ minimal, there is a $q\leq p$ which is $\PP$-generic for $M$, for all $M\in\cN$.

Now, $\MA^{1.5}_{\aleph_1}(\text{layered})$ is the forcing axiom for the class of $\aleph_{1.5}$-c.c.\ with respect to layered families, for meeting $\aleph_1$ dense sets. This forcing axiom implies that $\omega_1\pto[\omega_1]^2_{\bd}$. The proof is the same as presented by Uri.

In the paper of David and Miguel,\footnote{\href{https://doi.org/10.1007/s11856-015-1250-0}{Asper\'o and Mota, A generalization of Martin's axiom. (Isr.~J.~Math.\ 210(1) 193--231.)}} we can force this with continuum $\aleph_2$ without large cardinals. The natural extension of the question now is whether or not we can get the forcing axiom, or at least the polychromatic result, with a larger continuum? All that we know is that the polychromatic theorem is inconsistent with $\CH$.

\vspace{2em}
\textbf{Question 2:} Can you push the result under $\MM$ to get a polychromatic set of type $\omega_1+1$? What about higher cardinals, perhaps with higher forcing axioms?
\section{Example: Construction of a Boolean algebra of width \texorpdfstring{$\omega$}{w} and height \texorpdfstring{$\omega_2$}{w2} with 2-types side conditions (presented by David Asper\'o)}

\begin{definition}
Let $\kappa$ and $\lambda$ be two cardinals, typically $\kappa<\lambda$. We say that a partial order $\leq$ on $\kappa\times\lambda$ is a Locally Compact Scattered (LCS) space if the following properties hold:
\begin{enumerate}
\item If $(\alpha_0,\rho_0)<(\alpha_1,\rho_1)$, then $\rho_0<\rho_1$.
\item For all $x_0,x_1$, there is a finite set $b$ such that if $x\leq x_0$ and $x\leq x_1$, then there is some $y\in b$, then $x\leq y$. Such $b$ is called a \textit{barrier}.
\item For every $(\alpha,\rho)$ and every $\rho'<\rho$, then there are infinitely many $\beta$ such that $(\beta,\rho')<(\alpha,\rho)$.
\end{enumerate}
\end{definition}
Baumgartner and Shelah forced such a partial order on $\omega\times\omega_2$. Boban proved the same using side-conditions with Giorgio Venturi.\footnote{\href{https://arxiv.org/abs/1110.0610}{Velickovic and Venturi, Proper forcing remastered. (Appalachian set theory 2006--2012, 331--362.)}} David tries to guess the proof.
\begin{remark}
Why does the ``obvious'' approach using finite approximations of such a partial order fail? The reason is that on the $\omega_1$ level, with finite conditions, for every pair of points---and there are only $\omega$ of those---we can ensure the barrier is arbitrarily high below $\omega_1$, and this will code a surjection from $\omega$ onto $\omega_1$.
\end{remark}

We define our forcing, $\PP$, where $p\in\PP$ is a triplet $(\leq_p,b_p,\cN_p)$ such that:
\begin{enumerate}
\item $\leq_p$ is a finite approximation to the partial order.
\item $b_p$ is a finite approximation to the barrier function.
\item $\cN_p$ is a finite sequence of models of $2$-types. Namely, a sequence $(Q_i\mid i<n)$ such that $Q_i\prec H(\omega_2)$ with the following properties:
\begin{enumerate}
\item $Q_i$ is countable, or of size $\aleph_1$ and $\sigma$-closed.
\item $Q_i\in Q_{i+1}$ for all $i<n-1$.
\item For every $Q_i,Q_j$, there is some $k$ such that $Q_i\cap Q_j=Q_k$.
\item $b_p\restriction Q_i\in Q_i$ for all $i<n$.
\end{enumerate}
\end{enumerate}
We order $\PP$ by pointwise reverse inclusion. Then quite obviously, the generic will define the wanted partial order, but we need to verify that $\omega_1$ was not collapsed. We prove this by arguing that $\PP$ is strongly proper.
\begin{proposition}
$\PP$ is proper for $\sigma$-closed models of size $\aleph_1$. Namely, if $N^*\prec H(\theta)$ with $\PP\in N^*$ of size $\aleph_1$ and $\sigma$-closed, then any condition in $N^*\cap\PP$ can be extended to an $N^*$-generic condition.
\end{proposition}
\begin{proof}
Let $N^*$ be as above, and $p\in N^*$, then $q=(\leq_p,b_p,\cN_p\cup\{N^*\cap H(\omega_2)\})$ is $N^*$-generic. The verification here is straightforward.
\end{proof}
\begin{proposition}
$\PP$ is strongly proper.
\end{proposition}
\begin{proof}
Suppose that $M^*\prec H(\theta)$ is a countable elementary submodel with $\PP\in M^*$, and let $p\in\PP\cap M^*$. Write $M=M^*\cap H(\omega_2)$ and define $q$ to be $(\leq_p,b_p,\cN_q)$ where $\cN_q=\cN_p\cup\{M\}\cup\{N'\cap M\mid N'\in\cN_p\}$. We claim now that $q$ is $M^*$-generic.

Let $D\in M^*$ be a dense open set, let $q'\leq q$ be an extension in $D$. We define a projection map, $\pi(q')=(\leq_{q'}\cap M,b_{q'}\cap M,\cN_{q'}\cap M)$, then $\pi(q')\in\PP\cap M^*$. Let $r\leq\pi(q')$ be a condition in $D\cap M^*$. We define $t$ such that $\leq_t$ is the transitive closure (as an order) of $\leq_r\cup\leq_{q'}$, $b_t$ is the closure of $b_r\cup b_{q'}$, and $\cN_t$ is the closure of $\cN_r\cup\cN_{q'}$ under intersections. We can do a careful analysis to see that $b_t$ is indeed a barrier function for pairs of points coming from either $q'$ or $r$. the idea is that if we have two points coming from $q'$, outside of $M^*$, then there is a point below both of them, then it is either not in $M^*$, or it is below some other two points which are in $M^*$. In the former case, there is no problem and in the latter case, we simply use the fact that the barriers of the points in $M^*$ is below the original barrier. The case for points coming from $r$ is somehow similar.

The real difficulty is to define the new barriers for pairs of points which lie in $q'\setminus\pi(q)$ and $r\setminus\pi(q)$, and making sure that these come from the models in $\cN_t$. Here the idea is that if a point lies below both of the two points, there are points interpolating this where the barriers are already defined, so we take the union of these barriers as the barriers for the new pairs.\footnote{David writes this in more detail on the board, any missing details are solely this author's fault.}

Finally, we need to ensure that if $Q\in\cN_t$ and our new pair of points is in $Q$, then the barrier is also there. There are many details to verify here, and Boban remarks that we actually need to choose $r$ wisely to simplify the argument.
\end{proof}
\section{Impossibility theorems}
\subsection{One thing}
Boban present Shelah's construction of a counterexample to forcing theorems for uncountable cardinals.\footnote{\href{https://doi.org/10.1007/s00153-003-0208-9}{Shelah, Forcing axiom failure for any $\lambda>\aleph_1$. (Arch.\ Math.\ Log.\ 43(3), 285--295.)}}

The forcing $\PP$ has conditions $p$ such that $p$ is a pair of sequences of length $\alpha$, for some $\alpha<\omega_2$, $(S_\xi^p;C^p_\xi\mid\xi\leq\alpha)$ such that:
\begin{enumerate}
\item $S^p_\xi$ is stationary in $\omega_2$.
\item $C^p$ is an approximation for a $\square_{\omega_1}$ sequence.
\item $C^p_\xi\cap S^p_\xi=\varnothing$ for all $\xi$.
\item If $\zeta\in C^p_\xi$, then $S^p_\xi=S^p_\zeta$.
\end{enumerate}
The order is extension of the two sequences.

Shelah then claims that this forcing does not add bounded subsets of $\omega_1$, and preserves stationary sets of $\omega_2$. If we have a forcing axiom for meeting $\aleph_2$ dense sets, then we can meet all the dense sets of the form ``the condition has a sequence of length at least $\alpha$''. But now we can get quite a few of these, and this gives you a club-guessing kind of problem.
\subsection{Another thing}
Shelah proved the following theorem.
\begin{theorem}
Given any $\tup{C_\alpha\mid\alpha\in S^2_1}$ which is a club sequence (i.e.\ $C_\alpha$ is a club in $\alpha$), there is a sequence $\tup{f_\alpha\mid\alpha\in S^2_1}$ of functions such that $f_\alpha\colon C_\alpha\to 2$ for all $\alpha$ and for which there is no function $G\colon\omega_2\to 2$ such that for all $\alpha\in S^2_1$, $f_\alpha(\xi)=G(\xi)$ for all $\xi$ on a tail of $C_\alpha$.
\end{theorem}
And the forcing that proves this is quite nice, but it is not $\sigma$-lattice.
\section{Literature survey}
The very last session of the workshop was dedicated to a literature survey, from specific notes and papers, to simply theorems one should know about. This list is far from being complete, and many papers have surely been omitted. Nevertheless, we hope it will serve as a starting point for those interested in the topics of higher forcing axioms.

We already included several references in the footnotes of this work, and we will not include them in this list as well.

\subsection{Side conditions}
\begin{itemize}
\item Itay Neeman, \textbf{Forcing with sequences of models of two types}. \textit{Notre Dame J.\ Form.\ Log.} \textbf{55} (2014), no.~2, 265--298. \doi{10.1215/00294527-2420666}
\item Rahman Mohammadpour and Boban Velickovic, \textbf{Guessing models and the approachability ideal}. \textit{Preprint}, \textbf{\arxiv{1802.10125}} (2018).
\end{itemize}
\subsection{Baumgartner's Axioms}
\begin{itemize}
\item James E.\ Baumgartner, \textbf{Iterated forcing}. Surveys in set theory, 1--59,
\textit{London Math.\ Soc.\ Lecture Note Ser., 87} (1983), ed.\ A.R.D.\ Mathias.
\item Franklin D.\ Tall, \textbf{Some applications of a generalized Martin's axiom}. \textit{Topology Appl.} \textbf{57} (1994), no.~2-3, 215--248. \doi{10.1016/0166-8641(94)90051-5}
\item William Weiss, \textbf{The equivalence of a generalized Martin's axiom to a combinatorial principle}. \textit{J.\ Symbolic Logic} \textbf{46} (1981), no.~4, 817--821. \doi{10.2307/2273230}
\item Saharon Shelah, \textbf{A weak generalization of MA to higher cardinals}. \textit{Israel J.\ Math.} \textbf{30} (1978), no.~4, 297--306. \doi{10.1007/BF02761994}
\end{itemize}
\subsection{Symmetric systems}
\begin{itemize}
\item David Asper\'o and Miguel Angel Mota, \textbf{Forcing consequences of PFA together with the continuum large}. \textit{Trans.\ Amer.\ Math.\ Soc.} \textbf{367} (2015), no.~9, 6103--6129. \doi{10.1090/S0002-9947-2015-06205-9}
\end{itemize}
\subsection{Other type of forcing axioms}
\begin{itemize}
\item John Krueger and Miguel Angel Mota, \textbf{Coherent adequate forcing and preserving CH}. \textit{J.\ Math.\ Log.} \textbf{15} (2015), no.~2, 1550005, 34 pp. \doi{10.1142/S0219061315500051}
\item Gunter Fuchs, \textbf{Hierarchies of forcing axioms, the continuum hypothesis and square principles}. \textit{J.\ Symb.\ Log.} \textbf{83} (2018), no.~1, 256--282. \doi{10.1017/jsl.2017.46}
\item Gunter Fuchs, \textbf{The subcompleteness of Magidor forcing}. \textit{Arch.\ Math.\ Logic} \textbf{57} (2018), no.~3-4, 273--284. \doi{10.1007/s00153-017-0568-1}
\item David Chodounsk\'y and Jind\v{r}ich Zapletal, \textbf{Why Y-c.c.} \textit{Ann.\ Pure Appl.\ Logic} \textbf{166} (2015), no.~11, 1123--1149. \doi{10.1016/j.apal.2015.07.001}
\item Uri Abraham, \textbf{Lecture notes on the P-ideal dichotomy}. \textit{Winter School in Abstract Analysis} \textbf{2009}. \href{http://artax.karlin.mff.cuni.cz/~vernj1am/download/ws09/PID.pdf}{Link}
\end{itemize}
\subsection{Forcing axioms on uncountable cardinals}
\begin{itemize}
\item Dorottya Sziraki, \textbf{Colorings, Perfect Sets and Games on Generalized Baire Spaces}. Ph.D.\ dissertation, \textit{Central European University, Budapest, Hungary} (2018). \href{https://sierra.ceu.edu/record=b1400414}{Link}
\item Philipp Schlicht and Dorottya Sziraki, \textbf{Open graphs and hypergraps on definable subsets of generalized Baire spaces}. In preparation. \href{http://karagila.org/mehifox/CLMPST2019_Sziraki.pdf}{Slides}
\item Todd Eisworth, \textbf{On iterated forcing for successors of regular cardinals}. \textit{Fund.\ Math.} \textbf{179} (2003), no.~3, 249--266. \doi{10.4064/fm179-3-4}
\item Andrzej Ros{\L}anowski and Saharon Shelah, \textbf{Iteration of $\lambda$-complete forcing notions not collapsing $\lambda^+$}. \textit{Int.\ J.\ Math.\ Math.\ Sci.} \textbf{28} (2001), no.~2, 63--82. \doi{10.1155/S016117120102018X}
\end{itemize}
\subsection{Club guessing and uniformisation}
\begin{itemize}
\item Assaf Rinot, \textbf{The uniformization property for $\aleph_2$}. (2012). \href{https://blog.assafrinot.com/?p=4515}{Link}
\end{itemize}
\end{document}